\documentclass[a4paper, reqno]{amsart}

\usepackage[T1]{fontenc}
\usepackage{lmodern, microtype, amssymb}
\usepackage[scr=dutchcal]{mathalpha}
\usepackage[dvipsnames]{xcolor}
\usepackage[colorlinks=true, citecolor=RoyalBlue, linkcolor=black, urlcolor=RoyalBlue]{hyperref}
\usepackage{tikz-cd}

\microtypecontext{spacing=nonfrench}
\microtypesetup{stretch=20, shrink=15, spacing=true}
\tikzcdset{arrow style=tikz, diagrams={>={Computer Modern Rightarrow[length=5pt,width=4pt]}}}
\newtheorem{theorem}{Theorem}[section]
\newtheorem*{main-theorem}{Main theorem}
\newtheorem*{thm}{Theorem}
\newtheorem{proposition}[theorem]{Proposition}
\newtheorem{corollary}[theorem]{Corollary}
\newtheorem{lemma}[theorem]{Lemma}

\theoremstyle{definition}

\theoremstyle{remark}
\newtheorem{remark}[theorem]{Remark}

\DeclareMathOperator\bt{BT}
\DeclareMathOperator\BT{\mathbf{BT}}

\DeclareMathOperator\D{\mathbf{D}}
\DeclareMathOperator\Dsp{Dsp}
\DeclareMathOperator\End{End}
\DeclareMathOperator\Fil{Fil}

\DeclareMathOperator\Frac{Frac}
\DeclareMathOperator\Gal{Gal}
\DeclareMathOperator\Hom{Hom}
\DeclareMathOperator\Log{Log}

\DeclareMathOperator\REP{\mathbf{Rep}}

\DeclareMathOperator\T{T}

\DeclareMathOperator\W{W}
\DeclareMathOperator\wideg{wideg}

\def\AinfK{A_{\mathrm{inf},K}}
\def\BinfK{B_{\mathrm{inf},K}}
\def\Acris{A_\mathrm{cris}}
\def\Bcris{B_\mathrm{cris}}
\def\BdR{B_\mathrm{dR}}

\usepackage{changepage}
\usepackage{datetime, inconsolata}

\title{Integrality of height-one formal groups}
\author{Martin Debaisieux}
\address{Département de Mathématique, Université de Mons, 7000 Mons, Belgium}
\email{martin.debaisieux@umons.ac.be}
\thanks{The author was supported by the FRIA of the Fonds de la Recherche Scientifique -- FNRS}
\subjclass{\texttt{14L05}, \texttt{11S31}; \texttt{11S82}, \texttt{37P20}, \texttt{11F80}, \texttt{11F85}}
\date{\today}

\begin{document}

\begin{adjustwidth}{-1mm}{-1mm}
    \begin{abstract}
        Let $K$ be a finite extension of $\mathbb{Q}_p$. We prove that a one-dimensional formal group law over $K$ has integral coefficients if and only if its multiplication-by-$n$ endomorphisms have integral coefficients for all integers $n$, in the height-one case, that is when the multiplication by $p$ has Weierstrass degree $p$. The proof relies on $p$-adic Hodge theory.
    \end{abstract}
\end{adjustwidth}

\maketitle

\section{Introduction}

\noindent Fix a prime number $p$ and an algebraic closure $\overline{\mathbb{Q}}_p\!$ of $\mathbb{Q}_p$. Throughout this article, $K$ denotes a finite extension of $\mathbb{Q}_p$, with ring of integers $\mathcal{O}_K$, maximal ideal $\mathfrak{m}_K\!$ and absolute Galois group $G_K = \Gal(\overline{\mathbb{Q}}_p/K)$. All formal groups in this paper are one-dimensional commutative formal group laws; we refer the reader to \cite{Haz12} for background.

    \subsection{An integrality question} When
    \[
        F(X,Y) = X + Y + \cdots \in K[\![X,Y]\!]
    \]
    is a formal group with all its coefficients in $\mathcal{O}_K$, it follows that the multiplication-by-$n$ endomorphism $[n]_F$ of $F$ has integral coefficients for all $n \in \mathbb{Z}$. We investigate the converse: \emph{to what extent is the integrality of a formal group determined by that of all its endomorphisms of multiplication by an integer?} Basic computations show that polynomial and Lorentz formal groups are defined over $\mathcal{O}_K$ if and only if their multiplication-by-$2$ is. Here, we prove the following result.

    \begin{main-theorem}[\ref{main-theorem}]
        Let $F$ be a formal group defined over $K$. Assume that $[n]_F\!$ is defined over $\mathcal{O}_K$ for all $n \in \mathbb{Z}$ and that the Weierstrass degree of $[p]_F\!$ is $p$. Then, $F$ is defined over $\mathcal{O}_K$.
    \end{main-theorem}

    This result is perhaps not surprising: since the endomorphisms $[n]_F$ arise from repeated iterations of the formal power series $F$, one might expect them to determine its integrality. However, elementary techniques such as developing the coefficients of $[n]_F$ in terms of $F$, or a degree-by-degree analysis of the identities
    \[
        [n]_F(F(X,Y)) = F([n]_F(X),[n]_F(Y))
    \]
    for every $n \in \mathbb{Z}$, do not seem sufficient to conclude about the integrality of all the coefficients of $F$. Our proof instead makes essential use of the methods in \cite{Deb26} where such families of integral power series can be used to reconstruct their latent formal group, therefore showing it has integral coefficients.

    \subsection{A dynamical restatement}\label{S: dynamical-restatement} This question falls within the scope of Lubin's theory of $p$-adic dynamical systems, where one studies the relation between integral formal groups and families of integral formal power series without constant term that commute under composition (see \cite{Lub94}). Families of endomorphisms of an \makebox[\linewidth][s]{integral formal group are such. We call a power series $s \in X \mathcal{O}_K[\![X]\!]$ stable if $s'(0)$}

    \pagebreak

    \noindent is not zero nor a root of unity and say that a subset $D \subseteq X \mathcal{O}_K[\![X]\!]$ is stable if $D$ contains a stable power series. When $D$ is a stable commuting family defined over $\mathcal{O}_K$, Lubin showed the existence of a unique invertible formal power series
    \[
        \Log_D(X) \in XK[\![X]\!]
    \]
    such that $\Log_D'(0) = 1$ and $\Log_D(s(X)) = s'(0)\Log_D(X)$ for all $s \in D$, called the logarithm of $D$. If $D$ is a stable family of endomorphisms of a formal group, $\Log_D\!$ is the logarithm of the formal group. In any case,
    \[
        F_D(X,Y) = \Log_D^{\circ -1}(\Log_D(X) + \Log_D(Y))
    \]
    yields the unique formal group admitting the elements of $D$ as endomorphisms. A formal group can therefore be recovered from a stable family of endomorphisms,~but it is \textit{a priori} defined over $K$ rather than over $\mathcal{O}_K$. Our main theorem can thus be restated as follows.

    \begin{thm}[\ref{main-theorem}]
        Let $D = \{s_n(X) \in X \mathcal{O}_K[\![X]\!] \;;\; n \in \mathbb{Z}\}$ be a commuting family with $s_n'(0) = n$ for every $n \in \mathbb{Z}$. Assume that the Weierstrass degree of $s_p$ is $p$. Then,~its associated formal group $F_D$ is defined over $\mathcal{O}_K$, and $D \subset \End_{\mathcal{O}_K}(F_D)$.
    \end{thm}

    Proving that $F_D$ is integral given a certain stable family $D$ of commuting power series is an open problem known in the literature as Lubin's conjecture. Although partial results already exist for the families of Main Theorem (see \cite{Sar10}, \cite{Ber16}, \cite{HL16} and \cite{Spe18} over $\mathbb{Q}_p$; \cite{Ber19} over finite extensions of $\mathbb{Q}_p$ whose ramification index is smaller than $p$; and \cite{Deb26} over finite extensions of $\mathbb{Q}_p$ whose ramification index is coprime to $p(p-1)$), this paper gives a complete answer for these families. To this end, we follow the strategy of \cite{Deb26}. We study the dynamical system $D$ in order to recover a crystalline character of Hodge-Tate weight $1$. This property is characteristic of an integral formal group of height $1$. This allows us to endow the set of $s_p$-consistent sequences with a $\mathbb{Z}_p$-module structure for which the elements of $D$ are endomorphisms. Finally, we apply explicit functors in integral $p$-adic Hodge theory to show that $F_D$ is integral.

\section{The dynamical system}

\noindent Throughout this paper, we assume $D = \{s_n \;;\; n \in \mathbb{Z}\}$ to be a commuting family of elements in $X \mathcal{O}_K[\![X]\!]$ such that $s_n'(0) = n$ for all $n \in \mathbb{Z}$ and the Weierstrass degree of $s_p$ is $p$. Recall that the Weierstrass degree $\wideg(s)$ of $s \in \mathcal{O}_K[\![X]\!]$ is the $X$-adic valuation of its reduction $s(X) \bmod \mathfrak{m}_K$. Given that $\mathbb{Z}_p$ is the topological closure of $\mathbb{Z}$ in $K$, we may enlarge $D$ by
\[
    D = \{s_\alpha(X) \in X\mathcal{O}_K[\![X]\!] \;;\; \alpha \in \mathbb{Z}_p\}
\]
where $s_\alpha'(0) = \alpha$ and $s_\alpha \circ s_\beta = s_\beta \circ s_\alpha$ for all $\alpha$, $\beta \in \mathbb{Z}_p$ by \cite[Corollary 1.1.2.1]{Lub94}. The map sending $\alpha$ to $s_\alpha$ is a canonical bijection $\mathbb{Z}_p \rightarrow D$ of inverse $s_\alpha \mapsto s_\alpha'(0) = \alpha$ by \cite[Corollary 1.1.1]{Lub94}.

    \subsection{Torsion points}  Let $\mathbb{C}_p$ be the $p$-adic completion of $\overline{\mathbb{Q}}_p$. When we mention the roots or the fixed points of a power series, we mean its roots or its fixed points, respectively, in the open unit disk $\mathfrak{m}_{\mathbb{C}_p}\!$ of $\mathbb{C}_p$. A noninvertible stable $s \in X \mathcal{O}_K[\![X]\!]$ can have no other fixed points than $0$ but many roots. Let
    \[
        \Lambda(D) = \{x \in \mathfrak{m}_{\mathbb{C}_p}\! \mid \exists\, n \geqslant 0,\; s_p^{\circ n}(x) = 0\}.
    \]
    If $s \in D$, then $\Lambda(D)$ is also the set of roots of $s$ and its iterates. On the other hand, an invertible stable $\tilde{s} \in X \mathcal{O}_K[\![X]\!]$ can have no other roots than $0$, but many fixed points. If $\tilde{s} \in D$, Lubin showed in \cite[Proposition 3.2]{Lub94} that $\Lambda(D)$ is the set of periodic points of $\tilde{s}$. This set is therefore shared by all stable series in $D$, and is the set of torsion points of the latent formal group. In order to extract the latent Galois character attached to our dynamical system, we study the action of the group
    \[
        U = \{s_\alpha \in D \;;\; \alpha \in \mathbb{Z}_p^\times\} \simeq \mathbb{Z}_p^\times
    \]
    on the partition $\{\Lambda_n(D) \;;\; n \geqslant 0\}$ of $\Lambda(D)$, where $\Lambda_0(D) = \{0\}$ and, for all integers $n \geqslant 1$, $\Lambda_n(D)$ is the set of roots of $s_p^{\circ n}$ that are not roots of $s_p^{\circ n-1}$.

    \begin{proposition}\label{R: dynamical-study}
        Let $n \geqslant 1$ be an integer and $x \in \Lambda_n(D)$.
        \begin{enumerate}
            \item[(1)] For all $\alpha \in \mathbb{Z}_p^\times$, $s_\alpha(x) = x$ if and only if $\alpha \in 1+p^n\mathbb{Z}_p$.
            \item[(2)] $\Lambda_n(D) = \{s_\alpha(x) \;;\; \alpha \in \mathbb{Z}_p^\times\}$.
            \item[(3)] $\Lambda_n(D)$ has cardinality $p^n - p^{n-1}$.
            \item[(4)] $\upsilon_p(x) = 1/(p^n - p^{n-1})$.
            \item[(5)] For all $\alpha \in \mathbb{Z}_p^\times$, if $\upsilon_p(\alpha - 1) = n$ then $\wideg(s_\alpha(X) - X) = p^n$.
        \end{enumerate}
    \end{proposition}

    \begin{proof}
        By induction on $n$. Assume that $n = 1$ and let $x \in \Lambda_1(D)$. Then, $s_\alpha(x) = x$ implies that $x$ is a nonzero root of $s_\alpha(X) - X$, which thus can not be invertible and hence $\alpha - 1 \in \mathfrak{m}_K \cap \mathbb{Z}_p = p \mathbb{Z}_p$. Since $D$ is a commuting family,
        \[
            \{s_\alpha(x) \;;\; \alpha \in \mathbb{Z}_p^\times\} \subseteq \Lambda_1(D).
        \]
        The cardinality of the RHS set is at most $p - 1$ since $\wideg(s_p) = p$ and $0$ is a simple root of $s_p$. The LHS one has cardinality at least $(\mathbb{Z}_p^\times : 1+p \mathbb{Z}_p) = p - 1$ according~to the orbit-stabilizer theorem and what we have just shown. We deduce (2), (3) and (1), and thus that all the elements of $\Lambda_1(D)$ have the same valuation. The Newton polygon of $s_p$ starts at $(1,1)$, ends at $(p,0)$ and has a single line of slope $-1/(p-1)$. This implies (4). If $\upsilon_p(\alpha - 1) = 1$, then the Newton polygon of $s_\alpha(X) - X$ starts at $(1,1)$ and has a single segment related to $\Lambda_1(D)$, so must be equal to the Newton polygon of $s_p$. We deduce (5).

        Assume that the result holds for $n \geqslant 1$. If $\upsilon_p(\alpha - 1) \leqslant n$, then $s_\alpha(X) - X$ has at most $p^n$ roots by (5), contained in $\Lambda_0(D) \sqcup \cdots \sqcup \Lambda_n(D)$ by (1), and thus can not have roots outside of this set. Let $x \in \Lambda_{n+1}(D)$. If $s_\alpha(x) = x$, then $\alpha \in 1+p^{n+1}\mathbb{Z}_p$. One has an inclusion
        \[
            \{s_\alpha(x) \;;\; \alpha \in \mathbb{Z}_p^\times\} \subseteq \Lambda_{n+1}(D).
        \]
        The RHS set has cardinality at most $p^{n+1} - p^n$ since $\wideg(s_p^{\circ n+1}) = p^{n+1}$ and by (3). The LHS one has cardinality at least $(\mathbb{Z}_p^\times : 1+p^{n+1} \mathbb{Z}_p) = p^n(p-1)$ by what~we have just shown. We deduce (2), (3) and (1), and thus that the elements of $\Lambda_{n+1}(D)$ have the same valuation. Drawing the Newton polygon of $s_p(X) - s_p(x)$ shows that this valuation is $1/(p^{n+1} - p^n)$, hence (4). If $\upsilon_p(\alpha - 1) = n+1$, then the Newton polygon of $s_\alpha(X) - X$ starts at $(1, n+1)$, has a segment related to each $\Lambda_i(D)$ with $i \in \{1,\dots, n+1\}$, and thus must be equal to the Newton polygon of $s_p^{\circ n+1}$. This implies (5).
    \end{proof}

    \begin{remark}
        Consider a sequence $x = (x_n)_{n \geqslant 0}$ with components in $\mathfrak{m}_{\mathbb{C}_p}\!$ satisfying $x_0 = 0$ and $s_p(x_{n+1}) = x_n$ for all $n \geqslant 0$. Let $K_n = K(x_n)$ for every integer $n \geqslant 0$. This yields a tower of extensions of $K$. For all $n \geqslant 0$, the extension $K_n/K$ is Galois. Indeed, $x_n$ is a root of the Weierstrass polynomial of $s_p^{\circ n}$ which is defined over $K$, and every other root of this polynomial can be written in the form $s_\alpha(x_i) \in K(x_n)$ with $i \in \{0,\dots, n\}$ and $\alpha \in \mathbb{Z}_p^\times$. Given that $D$ is a commuting family, these Galois groups are abelian, and so is the extension $K(x)/K$.
    \end{remark}

    While simplicity of the roots is sometimes assumed in Lubin's conjecture when the commuting family is a pair (see for instance the statement on \cite[p. 131]{Sar05}), no such assumption is needed here. Indeed, point (3) of Proposition \ref{R: dynamical-study} implies by induction that the roots of $s_p$ and all its iterates are simple.

    \begin{corollary}\label{R: simple-roots}
        For all integers $n \geqslant 1$, the roots of $s_p^{\circ n}$ are simple.
    \end{corollary}

    \subsection{The logarithm}\label{S: logarithm} Since $D$ is a stable commuting family, let
    \[
        \Log_D(X) \in X+ X^2K[\![X]\!]
    \]
    be its logarithm (see Subsection \ref{S: dynamical-restatement}). The relation $\Log_D(s_\alpha(X)) = \alpha \Log_D(X)$ still holds for all $\alpha \in \mathbb{Z}_p$. It converges on $\mathfrak{m}_{\mathbb{C}_p\!}$, and its derivative has integral coefficients according to Corollary \ref{R: simple-roots} and \cite[Lemma 2.2]{Ber19}. The logarithm defines a formal group $F_D(X,Y)$ over $K$ for which the elements of $D$ are endomorphisms.

\section{The latent Tate module}

\noindent If the attached formal group $F_D$ is defined over $\mathcal{O}_K$, it has height $1$ and therefore $G_K$ acts on its $p$-adic Tate module by a crystalline character of weight $1$ with values in $\mathbb{Z}_p^\times$. The canonical isomorphism $\mathbb{Z}_p^\times \simeq U$ suggests that this character should be recoverable from the action of $U$. We proceed to show that it is indeed the case,~and that the resulting character enjoys the expected properties.

    \subsection{Recovering the Galois character} We first wish to imitate the action of~$G_K\!$ on the latent $p$-adic Tate module of $D$ using the one of $U$. Let $T_D$ be the $G_K$-set of $s_p$-consistent sequences, for which each element of $D$ is an endomorphism, acting by componentwise evaluation:
    \[
        T_D = \big\{(x_n)_{n \geqslant 0} \in \mathfrak{m}_{\mathbb{C}_p}^\mathbb{N}\! \;\big|\; x_0 = 0 \text{ and } \forall\, n \geqslant 0,\; s_p(x_{n+1}) = x_n \big\}.
    \]
    Let $S_D \subset T_D$ be the subset consisting of elements whose second entry is nonzero; it is stable under the action of $G_K$. The $n$-th coordinate of an element of $S_D$ lies in $\Lambda_n(D)$ and, for all $n \geqslant 0$ and $x \in \Lambda_n(D)$, there exists an element of $S_D$ whose $n$-th coordinate is $x$. After proving that $D$ is a family of endomorphisms of a height-one formal group defined over $\mathcal{O}_K$, it will follow that $T_D$ is the $p$-adic Tate module of this formal group, and is a free $\mathbb{Z}_p$-module of rank $1$. The set $S_D$ will be the set of its generators. Nonzero elements of $T_D$ are concatenation of a zero tuple with an element of $S_D$. In order to describe the action of $G_K$ on $T_D$, it is therefore sufficient to describe it on the subset $S_D$. We interpret the latter as a directed tree, as shown in Figure \ref{F: tree}.

    \begin{figure}[h]
        \begin{center}
            \begin{tikzcd}[column sep=-2.25mm, row sep=4mm, arrows=-]
                \vdots \ar[dr] & \vdots \ar[d] & \vdots \ar[dl] & \vdots \ar[dr] & \vdots \ar[d] & \vdots \ar[dl] & \vdots \ar[dr] & \vdots \ar[d] & \vdots \ar[dl] &
                \vdots \ar[dr] & \vdots \ar[d] & \vdots \ar[dl] & \vdots \ar[dr] & \vdots \ar[d] & \vdots \ar[dl] & \vdots \ar[dr] & \vdots \ar[d] & \vdots \ar[dl] &
                \vdots \ar[dr] & \vdots \ar[d] & \vdots \ar[dl] & \vdots \ar[dr] & \vdots \ar[d] & \vdots \ar[dl] & \vdots \ar[dr] & \vdots \ar[d] & \vdots \ar[dl] &&
                \vdots \ar[dr] & \vdots \ar[d] & \vdots \ar[dl] & \vdots \ar[dr] & \vdots \ar[d] & \vdots \ar[dl] & \vdots \ar[dr] & \vdots \ar[d] & \vdots \ar[dl] &
                \vdots \ar[dr] & \vdots \ar[d] & \vdots \ar[dl] & \vdots \ar[dr] & \vdots \ar[d] & \vdots \ar[dl] & \vdots \ar[dr] & \vdots \ar[d] & \vdots \ar[dl] &
                \vdots \ar[dr] & \vdots \ar[d] & \vdots \ar[dl] & \vdots \ar[dr] & \vdots \ar[d] & \vdots \ar[dl] & \vdots \ar[dr] & \vdots \ar[d] & \vdots \ar[dl] &\\
                & \ast \ar[drrr]\ar[rrr, densely dotted] &&& \ast \ar[d]\ar[rrr, densely dotted] &&& \ast \ar[dlll] &&& \ast \ar[drrr]\ar[rrr, densely dotted] &&& \ast \ar[d]\ar[rrr, densely dotted] &&& \ast \ar[dlll] &&& \ast \ar[drrr]\ar[rrr, densely dotted] &&& \ast \ar[d]\ar[rrr, densely dotted] &&& \ast \ar[dlll] &&&& \ast \ar[drrr]\ar[rrr, densely dotted] &&& \ast \ar[d]\ar[rrr, densely dotted] &&& \ast \ar[dlll] &&& \ast \ar[drrr]\ar[rrr, densely dotted] &&& \ast \ar[d]\ar[rrr, densely dotted] &&& \ast \ar[dlll] &&& \ast \ar[drrr]\ar[rrr, densely dotted] &&& \ast \ar[d]\ar[rrr, densely dotted] &&& \ast \ar[dlll]\\
                &&&& \ast \ar[drrrrrrrrr]\ar[rrrrrrrrr, densely dotted] &&&&&&&&& \ast \ar[d]\ar[rrrrrrrrr, densely dotted] &&&&&&&&& \ast \ar[dlllllllll] &&&&&&&&&& \ast \ar[drrrrrrrrr]\ar[rrrrrrrrr, densely dotted] &&&&&&&&& \ast \ar[d]\ar[rrrrrrrrr, densely dotted] &&&&&&&&& \ast \ar[dlllllllll]\\
                &&&&&&&&&&&&& \ast \ar[drrrrrrrrrrrrrr]\ar[rrrrrrrrrrrrrrrrrrrrrrrrrrrr,  densely dotted] &&&&&&&&&&&&&&&&&&&&&&&&&&&& \ast \ar[dllllllllllllll]\\
                &&&&&&&&&&&&&&&&&&&&&&&&&&& 0
            \end{tikzcd}
        \end{center}
        \caption{The arboreal interpretation of $S_D$ when $p = 3$.}
        \label{F: tree}
    \end{figure}

    \noindent Indexed from bottom to top, the vertices at level $n \geqslant 0$ are the elements of $\Lambda_n(D)$. For the sake of readability, we have omitted the arrowheads in the figure. There is an edge from a vertex $v_2$ to a vertex $v_1$ if and only if $s_p(v_2) = v_1$. Moreover, recall that
    \[
        \mathbb{Z}_p^\times \simeq \mathbb{F}_p^\times \times (1+p\mathbb{Z}_p)
    \]
    as topological groups. The horizontal connected components (dotted lines on the graph) at level $n \geqslant 1$ are the orbits in $\Lambda_n(D)$ under the action of $\{s_\alpha \in U \;;\; \alpha \in \mathbb{F}_p^\times\}$ if $n = 1$ and of $\{s_\alpha \in U \mid \upsilon_p(\alpha - 1) = n-1 \}$ if $n \geqslant 2$. This follows by (1) and (2) of Proposition \ref{R: dynamical-study}.

    \begin{corollary}\label{R: u-action-SD}
        The group $U\!$ acts simply transitively on the set $S_D$.
    \end{corollary}

    \begin{proof}
        By point (2) of Proposition \ref{R: dynamical-study}, the action is transitive on each level of $S_D$. By compactness of $\mathbb{Z}_p^\times$, these actions induce a transitive action on the limit $S_D$. It remains to show that this action is free. Let $x = (x_n)_{n \geqslant 0} \in S_D$. If $s_\alpha(x) = x$, then $s_\alpha(x_n) = x_n$ for all $n \geqslant 1$, and hence $\alpha \in 1+p^n \mathbb{Z}_p$ for every $n \geqslant 0$ by point (1) of Proposition \ref{R: dynamical-study}. This means that $\alpha = 1$. Therefore, $s_\alpha(X) = X$ by the uniqueness stated in \cite[Proposition 1.1]{Lub94}.
    \end{proof}

    We choose once and for all a sequence  $\pi = (\pi_n)_{n \geqslant 0} \in S_D$. By Corollary \ref{R: u-action-SD}, for every $g \in G_K$, there exists a unique $u_g \in U$ satisfying $g.\pi = u_g(\pi)$. This association defines a map
    \[
        \chi_D \colon G_K \longrightarrow \mathbb{Z}_p^\times \colon g \longmapsto u_g'(0)
    \]
    which is a character.

    \begin{proposition}
        The map $\chi_D$ is a character of $G_K\!$ satisfying $g.\pi = s_{\chi_D(g)}(\pi)$ for all $g \in G_K\!$ and is independent of the choice of $\pi$.
    \end{proposition}

    \begin{proof}
        The fact that this map is a character follows from the freeness of the action in Corollary \ref{R: u-action-SD}. The relation involving $\pi$ holds by construction. To show that $\chi_D$ is independent of $\pi$, consider $\pi' \in S_D$. By Corollary \ref{R: u-action-SD}, there exists a unique~$u \in U$ such that $\pi' = u(\pi)$, and then
        \[
            g.\pi' = g.u(\pi) = u(g.\pi) = u(u_g(\pi)) = u_g(u(\pi)) = u_g(\pi')
        \]
        because $U$ is a commuting family defined over $\mathcal{O}_K$. This shows that every element of $G_K$ acts on $\pi$ and $\pi'$ via the same power series in $U$.
    \end{proof}

    \subsection{Fontaine's period rings} Before determining the $p$-adic regularity of $\chi_D$, we recall some of the constructions in \cite{Fon94}, including the period rings $\Bcris$ and $\BdR$ that allow one to classify $p$-adic Galois representations. We also review the setting of formal groups.

        \subsubsection{The integers of the tilt of \texorpdfstring{$\mathbb{C}_p$}{Cp}} Let $\mathcal{O}_{\mathbb{C}_p}\!$ be the closed unit disk of $\mathbb{C}_p$. Throughout the remainder of this article, we fix $\pi_K$ a uniformizer of $K$ and we let $q$ be the cardinality of the residue field $\mathcal{O}_K/\mathfrak{m}_K = \mathbb{F}_q$. The starting point of the definition of the aforementioned period rings is the perfect ring
        \[
            \mathcal{R} = \varprojlim_{x \mapsto x^q} \mathcal{O}_{\mathbb{C}_p}/(\pi_K).
        \]
        This ring is equipped with a valuation $\upsilon_\mathcal{R}$, defined by $\upsilon_\mathcal{R}(x) = \lim_{n \to +\infty} q^n \upsilon_p(\widehat{x}_n)$ for all $x = (x_n)_{n \geqslant 0} \in \mathcal{R}$, where $\widehat{x}_n \in \mathcal{O}_{\mathbb{C}_p}$ lifts $x_n$ modulo $(\pi_K)$ for all integers $n \geqslant 0$ and where $\upsilon_p$ is the $p$-adic valuation normalized on $\mathbb{Q}_p$. It is also equipped with an action of $G_K$ inherited from the $G_K$-action on $\mathcal{O}_{\mathbb{C}_p\!}$.

        \subsubsection{Blank period rings} Let
        \[
            \AinfK = \W_{\!\pi_K}(\mathcal{R})
        \]
        be the ring of $\pi_K$-adic Witt vectors of $\mathcal{R}$. It is the unique strict $\pi_K$-ring of residue ring $\mathcal{R}$. Alternatively, one can see it as $\mathcal{O}_K \otimes_{\W(\mathbb{F}_q)} \W(\mathcal{R})$. Its elements are uniquely written in the form $\sum_{n \geqslant 0}[x_n]\pi_K^n$, where $x_n \in \mathcal{R}$ and $[-]$ is the Teichmüller lift. The action of $G_K$ lifts by functoriality, and there is a $G_K$-equivariant epimorphism
        \[
            \theta \colon \AinfK \longrightarrow \mathcal{O}_{\mathbb{C}_p}\!
        \]
        of $\mathcal{O}_K$-algebras determined by the requirement that $\theta([x]) = \lim_{n \to +\infty} \widehat{x}_n^{q^n}$, where $\widehat{x}_n \in \mathcal{O}_{\mathbb{C}_p}\!$ is a lift of the $n$-th component of $x \in \mathcal{R}$. By $K$-linearity, this map uniquely extends to a $G_K$-equivariant epimorphism $\theta$ from $\BinfK^+ = \AinfK[1/\pi_K]$ to $\mathbb{C}_p$ and its kernel is principal.

        \subsubsection{The de Rham period field} The de Rham period ring
        \[
            \BdR^+ = \varprojlim_{n \geqslant 1} \BinfK^+/\ker(\theta)^n
        \]
        is the $\ker(\theta)$-adic completion of $\BinfK^+$ and it inherits the action of $G_K$. Again, the map $\theta$ extends to it into a $G_K$-equivariant epimorphism onto $\mathbb{C}_p$ and its kernel is still principal. The $\mathbb{Z}$-powers of $\ker(\theta)$ induce a decreasing, exhaustive and separated filtration $(\Fil^n \BdR)_{n \in \mathbb{Z}}$ on the fraction field $\BdR$ of $\BdR^+$. The action of $G_K$ extends to this field; the elements fixed under this action are those in $K$. This field is our first ring of interest: we say that a $p$-adic representation $V$ of $G_K$ is de Rham if
        \[
            \dim_{\mathbb{Q}_p} V = \dim_K \Hom_{\mathbb{Q}_p[G_K]}(V, \BdR).
        \]
        If $(V, \rho)$ is a one-dimensional representation, being de Rham amounts to requiring a nonzero $z \in \BdR$ such that $g.z = \rho(g)z$ for all $g \in G_K$. Then, the Hodge-Tate~weight of $V$ is defined as the largest $w \in \mathbb{Z}$ satisfying $z \in \Fil^w \BdR$.

        \subsubsection{The crystalline period ring} Define $\Acris$ to be the $p$-adic completion of the PD-envelope of $\W(\mathcal{R})$ with respect to $\ker(\theta)$. This remains a subring of $\BdR$. Now, we set $\Bcris^+ = \Acris[1/p]$. The $G_K$-action on $\W(\mathcal{R})$ extends at each previous ring and the latter contains a unique $G_K$-stable $\mathbb{Q}_p$-line on which the group $G_K$ acts by the cyclotomic character. The crystalline period ring
        \[
            \Bcris = \Bcris^+[1/t]
        \]
        is obtained from $\Bcris^+$ by inverting any nonzero $t$ in this line. The action of $G_K\!$ naturally extends to this ring; the elements fixed under this action are those in $K_0$, the maximal unramified extension of $\mathbb{Q}_p$ inside $K$. The canonical homomorphism $K \otimes_{K_0} \Bcris \rightarrow \BdR$ is injective and $G_K$-equivariant. The crystalline period ring is of primary interest in this article: we say that a $p$-adic representation $V$ of $G_K$ is crystalline if
        \[
            \dim_{\mathbb{Q}_p} V = \dim_{K_0} \Hom_{\mathbb{Q}_p[G_K]}(V, \Bcris).
        \]
        The Hodge-Tate weights of a crystalline representation are its weights as a de Rham representation.

        \subsubsection{Application to height-one formal groups}\label{S: application-formal-groups} Given a height-one formal group $F$ over $\mathcal{O}_K$, one defines its $p$-adic Tate module by
        \[
            \T_p(F) = \varprojlim_{x \mapsto [p]_F(x)} F[p^n]
        \]
        where $[p]_F$ is the multiplication-by-$p$ endomorphism of $F$ and $F[p^n]$ is the group of $p^n$-torsion points of $F$ in $\mathfrak{m}_{\mathbb{C}_p}\!$ for every $n \geqslant 0$. Taking its $\mathbb{Q}_p$-scalar extension gives rise to a crystalline $G_K$-representation of dimension $1$ and Hodge-Tate weight $1$. In this paper, we are primarily interested in the integral setting for which the situation is as follows. Denote by $\mathbf{FGL}^1\!/\mathcal{O}_K$ the category of height-one formal groups defined over $\mathcal{O}_K$, and by $\mathbf{Rep}_{\mathbb{Z}_p}^{\mathrm{cris}, 1,1}(G_K)$ the category of $G_K$-stable $\mathbb{Z}_p$-lattices in crystalline $G_K$-representations of dimension $1$ and weight $1$. In \cite{Tat67}, Tate showed that the first category is equivalent to the category of height-one connected $p$-divisible groups over $\mathcal{O}_K$. Breuil in \cite[Theorem 5.3.2]{Bre00} if $p \neq 2$ and Kisin in \cite[Corollary 2.2.6]{Kis06} for any $p$ show that any object in the second one arises as the Tate module of a $p$-divisible group of aforementioned type. Therefore,
        \[
            \T_p \colon \mathbf{FGL}^1\!/\mathcal{O}_K \longrightarrow \mathbf{Rep}_{\mathbb{Z}_p}^{\mathrm{cris},1,1}(G_K)
        \]
        is an equivalence of categories. Hence, any endomorphism of a height-one formal group acts on the associated Tate module by multiplication by its linear coefficient, and the initial series is unequivocally determined by this action.

    \subsection{Regularity of the character} We now proceed to show that $\chi_D$ is crystalline of weight $1$. The key point lies in the construction of an additional ring, which is an adaptation of the ring $A_\infty$ in \cite{Spe18} for not necessarily Lubin-Tate formal groups. Let $r$ be the residue degree of $K/\mathbb{Q}_p$, so that $q = p^r$. By \cite[Corollary 6.2.1]{Lub94}, there exists an invertible formal power series $a(X) \in X\mathbb{F}_q[\![X]\!]$ such that
    \[
        f^{\circ r}(X) \bmod (\pi_K) = a(X^q).
    \]
    Let $A_\infty(\mathcal{O}_K)$ be the $\pi_K$-adic completion of the ring $\varinjlim \mathcal{O}_K[\![X_n]\!]$ with the transition maps $X_n \mapsto f^{\circ r}(X_{n+1})$ for all integers $n \geqslant 0$. We endow it with an action of $G_K$ by letting $g.X_n = s_{\chi_D(g)}(X_n)$ for all $g \in G_K$. Under the $(\pi_K, X_0)$-adic topology,~the $X_n$ are topologically nilpotent elements in $A_\infty(\mathcal{O}_K)$.

    \begin{lemma}
        The ring $A_\infty(\mathcal{O}_K)$ is a strict $\pi_K$-ring.
    \end{lemma}

    \begin{proof}
        The ring $A_\infty(\mathcal{O}_K)$ is a strict $\pi_K$-ring if it is $\pi_K$-adically complete, $\pi_K$-torsion free and $A_\infty(\mathcal{O}_K)/(\pi_K)$ is a perfect ring. Only the third one needs to be checked. For this, observe that
        \begin{align*}
            A_\infty(\mathcal{O}_K)/(\pi_K)
            & \simeq \varinjlim_{X_n \mapsto a(X_{n+1}^q)} \mathcal{O}_K[\![X_n]\!]/(\pi_K)\\
            & \simeq \varinjlim_{X_n \mapsto a(X_{n+1})^q} \mathbb{F}_q[\![X_n]\!]\\
            & \simeq \varinjlim_{a^{\circ n}(X_n) \mapsto a^{\circ n+1}(X_{n+1})^q} \mathbb{F}_q[\![a^{\circ n}(X_n)]\!]
        \end{align*}
        which is a perfect ring.
    \end{proof}

    The $f^{\circ r}$-consistent sequence $\varpi = (\varpi_n)_{n \geqslant 0}$ of topologically nilpotent elements in $\mathcal{R}$ given by
    \[
        \varpi_n = (a^{\circ m}(\pi_{(n+m)r} \bmod \pi_K))_{m \geqslant 0}
    \]
    for all integers $n \geqslant 0$ gives rise to a homomorphism $\phi_\varpi \colon A_\infty(\mathcal{O}_K) \rightarrow \mathcal{R}$ determined by the requirement that $\phi_\varpi(X_n) = \varpi_n$ for all $n \geqslant 0$. As $\pi_K$ is zero in $\mathcal{R}$, this map factors through $(\pi_K)$, and it is even its kernel.

    \begin{lemma}
        The homomorphism $\phi_\varpi$ is $G_K$-equivariant with kernel $\pi_K A_\infty(\mathcal{O}_K)$.
    \end{lemma}

    \begin{proof}
        The first assertion follows from the fact that $D$ is a commuting family. The second is the same as in \cite[Lemma 2.2]{Ber16}: if the induced map $\mathbb{F}_q[\![X_n]\!] \rightarrow \mathcal{R}$ were not injective, there would exist a nonzero $P(X_n) \in \mathbb{F}_q[X_n]$ such that $P(\overline{\phi_\varpi}(X_n)) = 0$ and this would imply that $\varpi_n$ is algebraic over $\mathbb{F}_p$, which is clearly not the case.
    \end{proof}

    Given that $\AinfK$ is a strict $\pi_K$-ring, the factorization of $\phi_\varpi$ through its kernel lifts to a $G_K$-equivariant embedding $A_\infty(\mathcal{O}_K) \hookrightarrow \AinfK$. Under this identification, we let $x_0$ be the image of $X_0$. Equivariance guarantees that $g.x_0 = s_{\chi_D(g)}(x_0)$ for all $g \in G_K$. Let
    \[
        t_D = \Log_D(x_0) \in \BdR^+.
    \]
    This element lies in $\BdR^+$ by \cite[Lemma 3.2]{Ber14} since $\Log_D$ converges on $\mathfrak{m}_{\mathbb{C}_p\!}$ and $\theta(x_0) \equiv 0 \bmod (\pi_K)$. The existence of $t_D$ implies the crystallinity of $\chi_D$.

    \begin{proposition}
        The character $\chi_D\!$ of $G_K\!$ is crystalline of Hodge-Tate weight $1$.
    \end{proposition}

    \begin{proof}
        This is the same as \cite[Proposition 2.6]{Deb26}. For the reader's convenience and due to the stature of this result, we repeat the proof. The previous convergence in $\BdR^+$ is with respect to the $\ker(\theta)$-adic topology. So, there exists some $m \geqslant 1$ for which $\theta(b_m x_0^m) = b_m\theta(x_0)^m = 0$ implying $\theta(x_0) = 0$, where $b_m \in K^\times$ is the $m$-th coefficient of the formal power series $\Log_D$. Therefore $\theta(t_D) = 0$, and hence $\chi_D$ has weight at least $1$. If it had weight at least $2$, the $\upsilon_\mathcal{R}$-valuation of $x_0 \bmod (\pi_K) = \varpi_0$ would be at least $2$. However, by point (4) of Proposition \ref{R: dynamical-study},
        \begin{align*}
            \upsilon_\mathcal{R}(\varpi_0)
            & = \lim_{m \to +\infty} q^m\upsilon_p(\pi_{mr})\\
            & = \lim_{m \to +\infty} p^{mr}\frac{1}{p^{mr-1}(p-1)}\\
            & = \frac{p}{p-1}
        \end{align*}
        is less than $2$ when $p$ is odd; a contradiction. For $p=2$, it suffices to compute the next component of $x_0$ and this leads to a contradiction. In particular $t_D$ is nonzero. For every $g \in G_K$,
        \[
            g.t_D = \Log_D(g.x_0) = \Log_D(s_{\chi_D(g)}(x_0)) = \chi_D(g)t_D
        \]
        and hence $\chi_D$ is de Rham of Hodge-Tate weight $1$. To show that $\chi_D$ is crystalline, first observe that $t_D \in K \otimes_{K_0} \Bcris^+$ according to \cite[\S 4.1.3]{Fon94} since $\theta(x_0) = 0$ and $\Log_D'(X) \in \mathcal{O}_K[\![X]\!]$. This guarantees the existence of a crystalline period because
        \begin{align*}
            \dim_{\mathbb{Q}_p} \mathbb{Q}_p(\chi_D)
            & = \dim_K \Hom_{\mathbb{Q}_p[G_K]}(\mathbb{Q}_p(\chi_D), K \otimes_{K_0} \Bcris)\\
            & = \dim_K K \otimes_{K_0} \Hom_{\mathbb{Q}_p[G_K]}(\mathbb{Q}_p(\chi_D), \Bcris)\\
            & = \dim_{K_0} \Hom_{\mathbb{Q}_p[G_K]}(\mathbb{Q}_p(\chi_D), \Bcris),
        \end{align*}
        which concludes the proof.
    \end{proof}

    \subsection{Reconstruction of the Tate module} Let $F$ be a height-one formal group defined over $\mathcal{O}_K$ whose $p$-adic Tate module is isomorphic to $\mathbb{Z}_p(\chi_D)$; see Subsection \ref{S: application-formal-groups} for details. In particular, the $G_K$-character associated to $\T_p(F)$ is $\chi_D$ and so the element $t_D$ is a shared de Rham period. We proceed to transport the structure of $\T_p(F)$ to $T_D$ in such a way that the elements of $D$ remain endomorphisms.


    \begin{proposition}
        \label{R: structure-Tate-module}
        There exists a $\mathbb{Z}_p$-module structure on the $G_K$-set $T_D$ for which it is a crystalline $G_K$-character of Hodge-Tate weight $1$ and $D \subseteq \End_{\mathbb{Z}_p[G_K]}(T_D)$.
    \end{proposition}

    \begin{proof}
        To ensure compatibility on the shared period $t_D$, we must identify $[\alpha]_F$, the multiplication-by-$\alpha$ endomorphism in $F$, with $s_\alpha$ for all $\alpha \in \mathbb{Z}_p$ given that $s_\alpha(\pi)$ is sent to $\Log_D(s_\alpha(x_0)) = \alpha t_D$ under our constructions. Let $\pi_F$ be a generator of the $\mathbb{Z}_p$-module $\T_p(F)$. We define $\tau \colon T_D \rightarrow \T_p(F)$ by letting
        \[
            \tau(s_\alpha(\pi)) = [\alpha]_F(\pi_F)
        \]
        for all $\alpha \in \mathbb{Z}_p$. Granting the decomposition of $T_D$ and according to Corollary \ref{R: u-action-SD}, it yields a well-defined bijection between $T_D$ and $\T_p(F)$. This map is $G_K$-equivariant because $G_K$ acts on $T_D$ and $\T_p(F)$ via the same character $\chi_D$. The result follows by transport of structure.
    \end{proof}

    \begin{remark}
        Any other choice of formal group over $\mathcal{O}_K$ whose $p$-adic Tate module is isomorphic to $\mathbb{Z}_p(\chi_D)$ would induce the same structure on $T_D$. A different choice of generator of $\T_p(F)$ would yield the same $\mathbb{Z}_p[G_K]$-module structure on $T_D$, up to permuting its elements. This is immaterial for our purpose, as the group law and the property that $D$ is a set of endomorphisms of it are unaffected.
    \end{remark}

\section{The latent formal group}

\noindent Starting from the data $(T_D, D)$ endowed with the structure of Proposition \ref{R: structure-Tate-module}, we now reconstruct the latent formal group using functors from integral $p$-adic Hodge theory. Recall that an element $s \in D$ acts on $T_D$ by componentwise evaluation or, equivalently, by multiplication by its linear coefficient, and the series is uniquely determined by the action of $\mathbb{Z}_p$ on $T_D$. We follow the path exposed in \cite[\S 3]{Deb26} which is sufficiently explicit to track the behavior of the elements of $D$ throughout the constructions, and ultimately to show that the formal group obtained is indeed the one for which the family $D$ is a set of endomorphisms, namely $F_D(X,Y)$ from \S\ref{S: dynamical-restatement}. This approach proceeds through the work of Kisin in \cite{Kis06} and \cite{Kis09} on Breuil-Kisin modules, then via the work of Zink in \cite{Zin01} and \cite{Zin02} on windows and displays, indicated by the dashed path in the diagram:
\[
    \begin{tikzcd}[row sep=20mm, column sep=13mm]
        \mathbf{FGL}^1\!/\mathcal{O}_K \ar[r, shift left=.6ex, "\T_p"]\ar[d, shift left=.6ex]
        & \REP_{\mathbb{Z}_p}^{\mathrm{cris},1,1}(G_K) \ar[r, shift left=.6ex, dashed, "\D_\mathfrak{S}^*"] \ar[l, shift left=.6ex]
        & \BT^{\mathrm{f},1}_\mathfrak{S}(\varphi)\ar[d, dashed, shift left=.6ex, "S \otimes_{\varphi, \mathfrak{S}} -"] \ar[l, shift left=.6ex]
        \\
        p\textbf{-}\mathbf{Div}^{\mathrm{c}, 1}\!/\mathcal{O}_K \ar[u, dashed, shift left=.6ex, "\text{comult.}"] \ar[rr, shift left=.6ex]
        & & S\textbf{-}\mathbf{Win}^1\!/\mathcal{O}_K \ar[ll, shift left=.6ex, dashed, "\mathrm{BT}\,\circ\,\mathrm{Dsp}"]\ar[u, shift left=.6ex]
    \end{tikzcd}
\]
where the categories, together with the functors, appearing therein will be defined progressively throughout this section. For the reader's convenience, we repeat these constructions.

    \subsection{The associated Breuil-Kisin module} Recall that we let $K_0$ be the maximal unramified extension of $\mathbb{Q}_p$ inside $K$ and write $\sigma$ its Frobenius. We have also fixed a uniformizer $\pi_K$ of $K$, and we denote by $E(X)$ its Eisenstein minimal polynomial over $K_0$. Let $K_\infty$ be the extension generated over $K$ by a compatible sequence of $p$-th roots of $\pi_K$ in $\mathfrak{m}_{\mathbb{C}_p}\!$. We define two subrings $\mathfrak{S} \subset \mathfrak{S}^\mathrm{ur}$ of $\W(\mathcal{R})$ as follows. First, let
    \[
        \mathfrak{S} = \W(\mathbb{F}_q)[\![X]\!]
    \]
    be the Kisin ring, endowed with a $\sigma$-semilinear Frobenius endomorphism $\varphi$ defined by $\varphi(X) = X^p$. According to \cite[\S2.1.1]{Kis06}, there exists an embedding $\mathfrak{S} \hookrightarrow \W(\mathcal{R})$ that preserves the Frobenius. Let $\mathscr{E}$ be the fraction field of the $p$-adic completion of $\mathfrak{S}[1/X]$. The previous embedding extends to an embedding $\mathscr{E} \hookrightarrow \W(\Frac(\mathcal{R}))$, and we denote by $\mathscr{E}^\mathrm{ur}$ the maximal unramified extension of $\mathscr{E}$ inside $\W(\Frac(\mathcal{R}))[1/p]$ and by $\mathcal{O}_{\mathscr{E}^\mathrm{ur}}$ its ring of integers. We then define
    \[
        \mathfrak{S}^\mathrm{ur} = \mathcal{O}_{\mathscr{E}^\mathrm{ur}} \cap \W(\mathcal{R}).
    \]
    Since the absolute Galois group of $K_\infty$ fixes $\mathfrak{S}$, it acts naturally on $\mathfrak{S}^\mathrm{ur}$.

        \subsubsection{Breuil-Kisin modules} A Breuil-Kisin module $(\mathfrak{M}, \varphi)$ is a finite free $\mathfrak{S}$-module $\mathfrak{M}$ equipped with a $\varphi_\mathfrak{S}$-semilinear Frobenius endomorphism $\varphi \colon \mathfrak{M} \rightarrow \mathfrak{M}$ whose cokernel of its $\mathfrak{S}$-linearization $1 \otimes \varphi \colon \mathfrak{S} \otimes_{\mathfrak{S}, \varphi} \mathfrak{M} \rightarrow \mathfrak{M}$ is killed by $E(X)$. Together with $\varphi$-equivariant $\mathfrak{S}$-linear maps, these objects form the category $\BT_\mathfrak{S}(\varphi)$. Kisin obtained in \cite{Kis06} a functor $\D_\mathfrak{S}^*$ yielding an equivalence of categories
        \[
            \D_\mathfrak{S}^* \colon \REP_{\mathbb{Z}_p}^{\mathrm{cris},1,1}(G_K) \longrightarrow \mathbf{BT}^{\mathrm{f},1}_\mathfrak{S}(\varphi)
        \]
        when restricting to the full subcategory of rank-one formal Breuil-Kisin modules of the Breuil-Kisin modules.

        \subsubsection{Construction} Let
        \[
            \mathfrak{M}_D = \mathbf{D}_\mathfrak{S}^*(T_D)
        \]
        be the rank-one formal Breuil-Kisin modules attached to $T_D$. We examine how the endomorphisms are transported. Set $M_D = \Hom_{\mathbb{Z}_p[G_{K_\infty}]}(T_D, \mathfrak{S}^\mathrm{ur})$, which is the free $\mathfrak{S}$-module of rank $1$ appearing in \cite[Lemma 2.1.10]{Kis06}. The action of an element $s \in D$ on $T_D$ extends to $M_D$ to be the multiplication by its linear coefficient: given a generator $\pi \mapsto e$ of $M_D$, one has that $s(\pi \mapsto e) = s'(0)(\pi \mapsto e)$ since $s(\pi) = s'(0)\pi$. The formal power series $s$ is uniquely determined by the action of $\mathbb{Z}_p \subset \mathfrak{S}$ on $M_D$. By \cite[Lemma 2.1.10]{Kis06}, we may assume that $\mathfrak{M}_D \subseteq M_D$ as $\mathfrak{S}$-modules; therefore the action of $s$ on $M_D$ restricts to $\mathfrak{M}_D$ and hence the formal power series $s$ can be recovered from this action.

        \subsection{The associated window} Let $S$ be the $p$-adic completion of the PD-envelope of $\W(\mathbb{F}_q)[X]$ with respect to the ideal $(E(X))$, and let $\Fil^1 S \subset S$ be the closure of the ideal generated by the divided powers of $E(X)$. As topological rings, one has an identification $S/\Fil^1 S \simeq \mathcal{O}_K$. By definition, $\mathfrak{S} \subset S$, and the Frobenius $\varphi$ on $\mathfrak{S}$ extends to a $\sigma$-semilinear Frobenius endomorphism $\varphi$ on the ring $S$.

            \subsubsection{\texorpdfstring{$S$}{S}-windows} Let $R \in \{\mathcal{O}_K/(p^n)\,;\, n \geqslant 0\}$ be a ring. An $S$-window $(W, W', \varphi)$ over $R$ is a finite projective $S$-module $W$ together with a submodule $W'$ containing $(\Fil^1 S)W$ and a $\sigma$-semilinear Frobenius endomorphism $\varphi \colon W \rightarrow W$ such that:
            \begin{enumerate}
                \item $W/W'$ is a projective $R$-module,
                \item $\varphi(W') \subseteq pW$ and $W$ is generated by $\varphi(W) + \varphi/p(W')$ as an $S$-module,
                \item since $pW \subseteq (1 \otimes \varphi)(\varphi^*(W))$, where $\varphi^*(W) = S \otimes_{S, \varphi} W$, the map $1 \otimes \varphi$ is injective and there exists a unique $S$-linear map $\psi \colon W \rightarrow \varphi^*(W)$ such that $(1 \otimes \varphi) \circ \psi \colon W \rightarrow W$ is the multiplication-by-$p$ map. Then, for $n$ sufficiently large,
                \[
                    \varphi^{n-1*}(\psi) \circ \varphi^{n-2*}(\psi) \circ \cdots \circ \psi \colon W \longrightarrow \varphi^{n*}(W)
                \]
                factors through $(p, \Fil^1 S)\varphi^{n*}(W)$.
            \end{enumerate}
            Together with $\varphi$-equivariant $S$-linear maps preserving the distinguished submodule, these objects form the category $S\textbf{-}\mathbf{Win}/R$. In the particular case of $R = \mathcal{O}_K$, the data of an $S$-window over $\mathcal{O}_K$ is equivalent to the data of a system $(W_n)_{n \geqslant 1}$, where $W_n$ is an $S$-window over $\mathcal{O}_K/(p^n)$ and is obtained from $W_{n+1}$ by base change for every integer $n \geqslant 1$ (see \cite[p. 500]{Zin01}). The extension
            \[
                S \otimes_{\varphi, \mathfrak{S}} - \colon \mathbf{BT}^{\mathrm{f}, 1}_\mathfrak{S}(\varphi) \longrightarrow S\textbf{-}\mathbf{Win}^1\!/\mathcal{O}_K
            \]
            defined in \cite[\S1.2.3]{Kis09} and restricted to rank-one formal Breuil-Kisin modules is an equivalence of categories, where $S\textbf{-}\mathbf{Win}^1\!/\mathcal{O}_K$ is the essential image of the LHS category inside $S\textbf{-}\mathbf{Win}/\mathcal{O}_K$ under the latter functor.

            \subsubsection{Construction} Let
            \[
                W_D = S \otimes_{\varphi, \mathfrak{S}} \mathfrak{M}_D
            \]
            be the object in $S\textbf{-}\mathbf{Win}^1\!/\mathcal{O}_K$ attached to our rank-one formal Breuil-Kisin module. Since Zink's functors in \cite{Zin01} for constructing $p$-divisible groups are formulated in terms of the limit description, we shall work with this perspective. Hence, let $(W_{D,n})_{n \geqslant 1}$ be the corresponding system. Upon passing to the quotient of $W_D$ with which $W_{D,n}$ is identified, we may assume that $W_{D,n}$ is a quotient of $W_D$. The action of $s \in D$ on $\mathfrak{M}_D$ extends to an action of $s$ on $W_D$ to be the multiplication-by-$s'(0)$ map. This action descends compatibly to each quotient $W_{D,n}$. Taking the $W_{D,n}$ as quotients of $W_D$ guarantees that it is still $s$ that acts on these objects, and $s$ can be uniquely recovered from the action of $\mathbb{Z}_p \subset S$ on them.

        \subsection{The associated \texorpdfstring{$p$}{p}-divisible group} Denote by $p\textbf{-}\mathbf{Div}^{\mathrm{c},1}\!/\mathcal{O}_K$ the category of height-one connected $p$-divisible groups over $\mathcal{O}_K$. In \cite{Zin01}, Zink introduced two functors, namely $\Dsp$ and $\bt$, such that
        \[
            \bt \circ \Dsp \colon S\textbf{-}\mathbf{Win}^1\!/\mathcal{O}_K \longrightarrow p\textbf{-}\mathbf{Div}^{\mathrm{c},1}\!/\mathcal{O}_K
        \]
        is an equivalence. These functors are defined on $S$-windows over $\mathcal{O}_K$ expressed as limits. Let $n \geqslant 1$ and $R_n = \mathcal{O}_K/(p^n)$. Set $P_{D,n} = \Dsp(W_{D,n}) = \W(R_n) \otimes_{\varkappa_n, S} W_{D,n}$ which is a $\W(R_n)$-module, with $\varkappa_n \colon S \rightarrow \W(R_n)$ the map of \cite[p. 499]{Zin01}. The action of $s \in D$ on $W_{D,n}$ extends to an action on $P_{D,n}$ to be the multiplication-by-its-linear-coefficient map. The formal power series $s$ remains uniquely recoverable from the action of $\mathbb{Z}_p \subseteq \W(R_n)$ on $P_{D,n}$. Then, set
        \[
            \Gamma_D = \bt(\Dsp(W_D)) = \varprojlim_{m \geqslant 1} \bt(P_{D,m})
        \]
        which is a height-one connected $p$-divisible group over $\mathcal{O}_K$. Let $\W_\mathrm{f}$ be the functor of Witt vectors with finite support, one defines $\Gamma_{D,n} = \bt(P_{D,n})$ as a quotient of the functor $\W_\mathrm{f}(-) \otimes_{\W(R_n)} P_{D,n}$. The action of an element $s \in D$ on $P_{D,n}$ extends to the previous functor, descends to the quotient $\Gamma_{D,n}$ and is compatible with the transition maps. Consequently, the action of $s$ passes to the limit $\Gamma_D$ and is, again, the multiplication by $s'(0)$. So, the formal power series $s$ remains an endomorphism at this step and is uniquely determined by the structure of $\Gamma_D$, which is induced by the one on $T_D$.

        \subsection{The main theorem} Following the path presented at the beginning of this section, we have constructed a height-one connected $p$-divisible group $\Gamma_D$ over $\mathcal{O}_K$, whose structure is inherited from that of $T_D$, and for which the elements of $D$ are endomorphisms. We may now state the main result.

        \begin{theorem}\label{main-theorem}
        The formal group $F_D$ is defined over $\mathcal{O}_K$, and $D = \End_{\mathcal{O}_K}(F_D)$.
        \end{theorem}

        \begin{proof}
            The comultiplication of the Hopf algebra of $\Gamma_D$ is a height-one formal group defined over $\mathcal{O}_K$ for which $D$ is a set of endomorphisms. Since $D$ is a stable family, its associated formal group is unique and hence is $F_D$. Given that $F_D$ has height $1$, its $\mathbb{Z}_p$-module of $\mathcal{O}_K$-endomorphisms is free of rank $1$, implying the equality with the enlarged family $D$.
        \end{proof}

        \begin{remark}
            Thanks to the decomposition of $\mathbb{Z}_p$, only three series (such as $s_\alpha$ with $\alpha \in \{p, 1+p, \zeta\}$, where $\zeta \in \mathbb{Z}_p^\times$ is a generator of $\mathbb{F}_p^\times$) are needed to generate all the nonzero elements of $D$, and hence to carry out all our constructions.
        \end{remark}

    \section*{Acknowledgement}

    \noindent The author would like to express his gratitude to his supervisor Maja Volkov for her thoughtful guidance and constructive feedback throughout this work.

\end{document}